\documentclass[11pt]{amsart}
\usepackage{a4wide}

\usepackage[all]{xy}

\makeatletter
\@namedef{subjclassname@2010}{%
  \textup{2010} Mathematics Subject Classification}
\makeatother

\newtheorem{theorem}{Theorem}[section]
\newtheorem{problem}[theorem]{Problem}
\newtheorem{proposition}[theorem]{Proposition}

\newtheorem{question}[theorem]{Question}
\newtheorem{lemma}[theorem]{Lemma}
\newtheorem{claim}[theorem]{Claim}

\theoremstyle{definition}
\newtheorem{example}[theorem]{Example}

\newtheorem{remark}[theorem]{Remark}

\newcommand{\w}{\omega}

\newcommand{\Ra}{\Rightarrow}
\newcommand{\F}{\mathcal F}
\newcommand{\U}{\mathcal U}
\newcommand{\V}{\mathcal V}
\newcommand{\IQ}{\mathbb Q}
\newcommand{\IR}{\mathbb R}
\newcommand{\IZ}{\mathbb Z}
\newcommand{\Tau}{\mathsf T\!}
\newcommand{\Wallman}{\mathsf W}
\newcommand{\II}{\mathbb I}
\newcommand{\Wk}{\mathsf W_{\!\bar \kappa}}

\begin{document}

%%%%% To ease editing, for IMPAN journals add:

%\baselineskip=16pt
%\parskip=1pt

%%%%%%%%%%%%%%%%

\title[Embedding topological spaces into $\kappa$-bounded spaces]{Embedding topological spaces into\\ Hausdorff $\kappa$-bounded spaces}

\author[T. Banakh]{Taras Banakh}
\address{Ivan Franko National University of Lviv, Ukraine\\
and Jan Kochanowski University in Kielce, Poland}
\email{t.o.banakh@gmail.com}

\author[S. Bardyla]{Serhii Bardyla}
\address{Institute of Mathematics, Kurt G\"{o}del Research Center\\
University of Vienna, Austria}
\email{sbardyla@yahoo.com}
\thanks{The work of the second author is supported by the Austrian Science Fund FWF (Grant  I 3709 N35).}

\author[A. Ravsky]{Alex Ravsky}
\address{Pidstryhach Institute for Applied Problems of Mechanics and Mathematics, Lviv, Ukraine}
\email{alexander.ravsky@uni-wuerzburg.de}

\date{}

\begin{abstract}
Let $\kappa$ be an infinite cardinal. A topological space $X$ is {\em $\kappa$-bounded} if the closure of any subset of cardinality $\le\kappa$ in $X$ is compact. We discuss the problem of embeddability of topological spaces  into Hausdorff (Urysohn, regular) $\kappa$-bounded spaces, and present a canonical construction of such an embedding. Also we construct a (consistent) example of a sequentially compact separable regular space that cannot be embedded into a Hausdorff $\omega$-bounded space.
\end{abstract}



\maketitle

\section{Introduction}

It is well-known that a topological space $X$ is homeomorphic to a subspace of a compact Hausdorff space if and only if the space $X$ is Tychonoff.  In this paper we address the problem of characterization of topological spaces that embed into Hausdorff (Urysohn, regular, resp.) spaces possessing some weaker compactness properties.

One of such properties is the $\kappa$-boundedness, i.e., the compactness of closures of subsets of cardinality $\le\kappa$.
%A topological space $X$ is {\em $\kappa$-bounded} for a cardinal $\kappa$ if each subset $A\subset X$ of cardinality $|A|\le\kappa$ has compact closure in $X$.
It is clear that each compact space is $\kappa$-bounded for any cardinal $\kappa$. Any ordinal $\alpha:=[0,\alpha)$ of cofinality $\mathrm{cf}(\alpha)>\kappa$, endowed with the order topology, is $\kappa$-bounded but not compact. More information on $\kappa$-bounded spaces can be found in \cite{JvM}, \cite{JSS}, \cite{JV}. Embedding of topological spaces into compact-like spaces was also  investigated in~\cite{Ban, BBR, Dik, Dik1}.

In this paper we discuss the following:

\begin{problem}\label{prob1} Which topological spaces are homeomorphic to subspaces of $\kappa$-bounded Hausdorff (Urysohn, regular) spaces?
\end{problem}

In Theorem~\ref{t:main} (and Theorem~\ref{t:main1}) we shall prove that the necessary and sufficient conditions of embeddability of a $T_1$-space $X$ into a Hausdorff (Urysohn) $\kappa$-bounded space are the (strong) $\overline\kappa$-regularity and the (strong)  $\overline\kappa$-normality of $X$, respectively.
In Theorem~\ref{t:main2} we shall prove that a sufficient condition of embeddability of a $T_1$-space $X$ into a regular $\kappa$-bounded space is the total $\overline\kappa$-normality of $X$.
The above mentioned separation axioms are introduced and studied in Section~\ref{s:norm}. In Section \ref{s:W} we shall present a canonical construction of an embedding a (strongly or totally) $\overline{\kappa}$-normal space into a Hausdorff (Urysohn or regular) $\kappa$-bounded space. In Section~\ref{s:ex} we construct a space that is totally $\overline{\omega}$-normal but not functionally Hausdorff, and a (consistent) example of a sequentially compact separable regular space which is not Tychonoff and hence does not embed into a Hausdorff $\w$-bounded space. Also, for each cardinal $\kappa$ we construct a topological space which is $\kappa$-bounded, $\overline{\kappa}$-normal, H-closed, but not Urysohn.
%All topological spaces appearing in this paper are assumed to satisfy the separation axiom $T_1$.

\section{Useful separation axioms}\label{s:norm}
%For a topological space $X$ by $2^{X}$ we denote the family of all closed subsets of $X$.
Let $\F$ be a family of closed subsets of a topological space $X$.  The topological space $X$ is called
\begin{itemize}
\item {\em $\F$-regular} if for any set $F\in\F$ and point $x\in X\setminus F$ there exist disjoint open sets $U,V\subset X$ such that $F\subset U$ and $x\in V$;
\item {\em strongly $\F$-regular} if for any set $F\in\F$ and point $x\in X\setminus F$ there exist open sets $U,V\subset X$ such that $F\subset U$, $x\in V$ and $\overline{U}\cap \overline{V}=\emptyset$;
%\item {\em strongly $\F$-regular} if for any set $F\in\F$ and point $x\in X\setminus F$ there exist open sets $U,V\subset X$ such that $F\subset U$ and $x\in V$;
\item {\em $\F$-Tychonoff} if for any set $F\in\F$ and point $x\in X\setminus F$ there exist a continuous function $f:X\to [0,1]$ such that $f(F)\subset \{0\}$ and $f(x)=1$;
\item {\em $\F$-normal} if for any disjoint sets $A,B\in\F$ there exist disjoint open sets $U,V\subset X$ such that $A\subset U$ and $B\subset V$;
\item {\em strongly $\F$-normal} if for any disjoint sets $A,B\in\F$ there exist open sets $U,V\subset X$ such that $A\subset U$, $B\subset V$ and $\overline{U}\cap \overline{V}=\emptyset$;
\item {\em totally $\F$-normal} if for any disjoint closed sets $A\in\F$ and $B\subset X$ there exist disjoint open sets $U,V\subset X$ such that $A\subset U$ and $B\subset V$.
\end{itemize}
For families $\F$ containing all one-point subsets of $X$, these properties relate as follows:
$$\xymatrix{
\mbox{totally $\mathcal{F}$-normal}\ar@{=>}[r]&\mbox{strongly $\mathcal{F}$-normal}\ar@{=>}[r]\ar@{=>}[d]&\mbox{$\mathcal{F}$-normal}\ar@{=>}[d]\\
\mbox{$\mathcal{F}$-Tychonoff}\ar@{=>}[r]&\mbox{strongly $\mathcal{F}$-regular}\ar@{=>}[r]&\mbox{$\mathcal{F}$-regular}.
}$$
%It is easy to see that for a topological $T_1$-space the following implications hold strongly $\F$-normal$\rightarrow$strictly $\F$-normal$\rightarrow\F$-normal$\rightarrow\F$-regular and $\F$-Tychonoff$\rightarrow\F$-regular.
However, the total $\F$-normality does not imply the $\F$-Tychonoff property; see Example~\ref{ex} below.

\begin{proposition}\label{p:rn} If a topological space $X$ is $\F$-regular for some family $\F$ of closed Lindel\"of subspaces of $X$, then $X$ is $\F$-normal.
\end{proposition}

\begin{proof} To show that $X$ is $\F$-normal, fix any two disjoint closed sets $A,B\in\F$. By the $\F$-regularity, for every $a\in A$ there exists an open neighborhood $V_a\subset X$ of $a$ whose closure $\overline{V}_a$ in $X$ does not intersect the set $B$. By the Lindel\"of property of $A$ the open cover $\{V_a:a\in A\}$ of $A$ has a countable subcover $\{V_{a_n}\}_{n\in\w}$.

By analogy, for every $b\in B$ there exists an open neighborhood $U_b\subset X$ of $b$ whose closure $\overline{U}_b$ in $X$ does not intersect the set $A$. By the Lindel\"of property of $B$, the open cover $\{U_b:b\in B\}$ of $B$ has a countable subcover $\{U_{b_n}\}_{n\in\w}$. For every $n\in\w$ let $$V_A=\bigcup_{n\in\w}V_{a_n}\setminus\bigcup_{k\le n}\overline{U}_{b_k}\mbox{ \ and \ }U_B=\bigcup_{n\in\w}U_{b_n}\setminus\bigcup_{k\le n}\overline{V}_{a_k}.$$
Then $V_A,U_B$ are two disjoint open neighborhoods of the sets $A,B$, witnessing that the space $X$ is $\F$-normal.
\end{proof}
%Analogously, one can prove the following proposition.
%\begin{proposition}\label{p:rn1} If a topological space $X$ is strictly $\F$-regular for some family $\F$ of closed Lindel\"of subspaces of $X$, then $X$ is strictly $\F$-normal.
%\end{proposition}

A subset $Y$ of a topological space $X$ is defined to be ({\em countably}) {\em paracompact in} $X$ if for each (countable) cover $\mathcal{U}$  of $Y$ by open subsets of $X$, there exists a locally finite family $\mathcal{V}$ of open subsets of $X$ such that $Y\subset\bigcup\mathcal V$ and each set $V\in\mathcal V$ is contained in some set $U\in\mathcal U$.

The proof of the following proposition is stratightforward.

\begin{proposition}\label{cp}
If a  subset $Y$ of a topological space $X$ is (countably) paracompact in $X$, then  each closed subset of $Y$ also is (countably) paracompact in $X$.
\end{proposition}

\begin{proposition}\label{p:rn1} Let $\F$ be a family of closed Lindel\"of subsets of $X$, which are countably paracompact in $X$. If the space $X$ is strongly $\F$-regular, then $X$ is strongly $\F$-normal.
\end{proposition}

\begin{proof}
To show that $X$ is strongly $\F$-normal, fix any two disjoint closed sets $A,B\in\F$. By the strong $\F$-regularity of $X$, for every $a\in A$ and $b\in B$ there exist open sets $V_a,W_a,V_b,W_b$ in $X$ such that $a\in V_a\subset \overline V_a\subset W_a\subset \overline W_a\subset X\setminus B$ and
$b\in V_b\subset \overline V_b\subset W_b\subset \overline W_b\subset X\setminus A$. 
By the Lindel\"of property of the space $A$, the open cover $\{V_a:a\in A\}$ of $A$ has a  countable subcover $\{V_{a_n}\}_{n\in\w}$. By analogy, the open cover $\{V_b:b\in B\}$ of the Lindel\"of space $B$ has a countable subcover $\{V_{b_n}\}_{n\in\w}$.

It is easy to see that $\{V_{a_n}\setminus\bigcup_{k\le n}\overline W_{b_k}\}_{n\in\w}$ is a cover of $A$ by open subsets of $X$. By the countable paracompactness of $A$, there exists a locally finite family $\U_A$   of open sets in $X$ such that $A\subset \bigcup\U_A$ and each set $U\in\U_A$ is contained in some set $V_{a_n}\setminus\bigcup_{k\le n}\overline W_{b_k}$ and hence $\overline U\subset \overline V_{a_n}\setminus \bigcup_{k\le n}W_{b_k}\subset W_{a_n}\setminus\bigcup_{k\le n}W_{b_k}$.

Consider the open neighborhood $U_A=\bigcup\U_A$ of $A$. The local finiteness of the family $\U_A$ ensures that $\overline{U_A}=\bigcup_{U\in\U_A}\overline{U}\subset \bigcup_{n\in\w}W_{a_n}\setminus\bigcup_{k\le n}W_{b_k}$.

By analogy, we can find an open neighbrhood $U_B$ of the countably paracompact subset $B$ in $X$ such that $\overline{U_B}\subset\bigcup_{n\in\w}W_{b_n}\setminus\bigcup_{k\le n}W_{a_k}$.

It remains to observe that
$$\overline{U_A}\cap\overline{U_B}\subset\Big(\bigcup_{n\in\w}W_{a_n}\setminus\bigcup_{k\le n}W_{b_k}\Big)\cap\Big(\bigcup_{n\in\w}W_{b_n}\setminus\bigcup_{k\le n}W_{a_k}\Big)=\emptyset.$$
\end{proof}

\begin{proposition}\label{p:rn2} Each regular topological space $X$ is totally $\F$-normal for the family $\F$ of closed subsets of $X$ that are paracompact in $X$.
\end{proposition}

\begin{proof}
To show that $X$ is totally $\F$-normal, fix any two disjoint closed sets $A\in\F$ and $B\subset X$. By the regularity of $X$, for every $a\in A$ there exists an open neighborhood $U_a\subset X$ such that $\overline{U_a}\subset X\setminus B$. Since $A$ is paracompact in $X$ there exists a locally finite family $\mathcal{V}$ of open subsets of $X$ such that $A\subset \bigcup\V$ and each set $V\in\V$ is contained in some set $U_a$, $a\in A$. The locally finiteness of $\mathcal{V}$ implies that $\overline{\bigcup \mathcal{V}}=\bigcup_{V\in\V}\overline{V}\subset \bigcup_{a\in A}\overline{U_a}\subset X\setminus B$. Then $\bigcup\V$ and $X\setminus\overline{\bigcup\V}$ are disjoint open neighborhoods of the sets $A$ and $B$, respectively.
\end{proof}

Let $\kappa$ be a cardinal. A topological space $X$ is called {\em totally $\overline{\kappa}$-normal} (resp. {\em strongly $\overline{\kappa}$-normal}, {\em $\overline{\kappa}$-normal}, {\em strongly $\overline{\kappa}$-regular}, {\em $\overline{\kappa}$-regular}, {\em $\overline\kappa$-Tychonoff\/}) if it is totally $\F$-normal (resp. strongly $\F$-normal, $\F$-normal, strongly $\F$-regular, $\F$-regular, $\F$-Tychonoff) for the family $\F$ of closed subsets of the closures of subsets of cardinality $\le\kappa$ in $X$. Simple examples show that the family $\F$ can be strictly larger than the family of closures of subsets of cardinality $\le \kappa$ in $X$.

\begin{proposition}\label{p:cn} Each $\kappa$-bounded Hausdorff space $X$ is $\overline{\kappa}$-normal.
\end{proposition}

\begin{proof} Let $\F$ be the family of closed subspaces of the closures of subsets of cardinality $\le\kappa$ in $X$. Given two disjoint closed sets $A,B\in\F$, we observe that the sets $A,B$ are compact. By the Hausdorff property of $X$, the disjoint compact sets $A,B$ have disjoint open neighborhoods.
\end{proof}

In Example~\ref{ex1} we shall construct a Hausdorff $\omega$-bounded space which is not strongly $\overline{\omega}$-normal.

\begin{proposition}\label{c1} Each subspace $X$ of a $\kappa$-bounded Hausdorff space $Y$ is $\overline{\kappa}$-regular.
\end{proposition}

\begin{proof} Let $F$ be a closed subspace of the closure  of a set $E\subset X$ of cardinality $|E|\le\kappa$ in $X$ and let $x\in X\setminus F$ be a point. The $\kappa$-boundedness of $Y$ ensures that the closure $\overline E$ of $E$ in $Y$ is compact and so is the closure $\overline F$ of $F\subset\overline E$ in $Y$.
Since $F=X\cap \overline F$ and $x\in X\setminus F$, $x\not\in\overline F$ and so  by
the Hausdorff property of $Y$ there exist two disjoint open sets $V,U\subset Y$ such that $x\in V$
and $\overline F\subset U$. Then $V\cap X$ and $U\cap X$ are two disjoint open sets in $X$ such
that $x\in V\cap X$ and $F\subset U\cap X$, which means that the space $X$ is
$\overline{\kappa}$-regular.
\end{proof}

Let us recall that a topological space $X$ is called {\em Urysohn} if any distinct points of $X$ have disjoint closed neighborhoods  in $X$.
Similarly as above one can prove the following facts.
\begin{proposition}\label{p:cn1} Each $\kappa$-bounded Urysohn space $X$ is strongly $\overline{\kappa}$-normal.
\end{proposition}

\begin{proposition}\label{c11} Each subspace $X$ of a $\kappa$-bounded Urysohn space $Y$ is strongly $\overline{\kappa}$-regular.
\end{proposition}

\begin{proposition}\label{p:cn2} Each $\kappa$-bounded regular space $X$ is totally $\overline{\kappa}$-normal.
\end{proposition}

We recall that the {\em density} $d(X)$ of a topological space $X$ is the smallest cardinality of a dense subset in $X$.

\begin{proposition}\label{p:Tych} Each subspace $X$ of density $d(X)\le\kappa$ in a $\kappa$-bounded Hausdorff space $Y$ is Tychonoff.
\end{proposition}

\begin{proof} Let $D$ be a dense subset of $X$ with $|D|=d(X)\le\kappa$. By definition of a $\kappa$-bounded space, the closure $\overline{D}$ of $D$ in $Y$ is compact and being Hausdorff is Tychonoff. Then $X\subset\overline{D}$ is Tychonoff, too.
\end{proof}

\section{The Wallman $\kappa$-bounded extension of a topological space}\label{s:W}

We recall \cite[\S3.6]{Eng} that the Wallman extension $\Wallman X$ of a topological space $X$ consists of closed ultrafilters, i.e., families $\U$ of closed subsets of $X$ satisfying the following conditions:
\begin{itemize}
\item $\emptyset\notin\U$;
\item $A\cap B\in\U$ for any $A,B\in\U$;
\item a closed set $F\subset X$ belongs to $\mathcal U$ if $F\cap A\ne\emptyset$ for every $A\in\U$.
\end{itemize}
The Wallman extension $\Wallman X$ of $X$ carries the topology generated by the base consisting of the sets
$$\langle U\rangle=\{\F\in \Wallman X:\exists F\in\F\;\;(F\subset U)\}$$ where $U$ runs over open subsets of $X$.

By (the proof of) Theorem~\cite[3.6.21]{Eng}, the Wallman extension $\Wallman X$ is compact. By Theorem~\cite[3.6.22]{Eng} a $T_1$-space $X$ is normal if and only if its Wallman extension $\Wallman X$ is Hausdorff. %{\color{blue}We recall that a} topological space $X$ is {\color{blue}$T_1$} if each finite subset of $X$ is closed in $X$.

If $X$ is a $T_1$-space, then we can consider the map $j_X:X\to \Wallman X$ assigning to each point $x\in X$ the principal closed ultrafilter consisting of all closed sets $F\subset X$ containing the point $x$. It is easy to see that the image $j_X(X)$ is dense in $\Wallman X$. By \cite[3.6.21]{Eng}, the map $j_X:X\to \Wallman X$ is a topological embedding. %So, $X$ can be identified with the subspace $j_X(X)$ of $\Wallman X$.

The following lemma can be easily derived from the definition of a closed ultrafilter and should be known.

\begin{lemma}\label{W} For a subset $A$ of a $T_1$-space $X$, a closed ultrafilter $\mathcal{F}\in \Wallman X$ belongs to $\overline{j_X(A)}\subset \Wallman X$ if and only if $\overline{A}\in\F$.
\end{lemma}

Given an infinite cardinal $\kappa$, in the Wallman extension $\Wallman X$ of a $T_1$-space $X$, consider the subspace $$\Wk X={\textstyle\bigcup}\{\overline{j_X(C)}:C\subset X,\;|C|\le\kappa\}$$of $\Wallman X$. The space $\Wk X$ will be called the {\em Wallman $\kappa$-bounded extension} of $X$.

The following proposition justifies the choice of terminology.

\begin{proposition}\label{t:w} For any topological space $X$, the space $\Wk X$ is $\kappa$-bounded.
\end{proposition}

\begin{proof} We should prove that for any subset $\Omega\subset \Wk X$ of cardinality $|\Omega|\le\kappa$, the closure $\overline{\Omega}$ is compact. By the definition of $\Wk X$, for every ultrafilter $u\in\Omega$ there exists a  set $C_u\subset X$ such that $|C_u|\le\kappa$ and $u\in\overline{j_X(C_u)}$. Consider the set $C=\bigcup_{u\in\Omega}C_u$ and observe that $|C|\le\kappa$ and the closure $\overline{j_X(C)}$ in $\Wallman X$ is compact (by the compactness of $\Wallman X$). Then the closure $\overline{\Omega}$ of $\Omega$ in $\Wk X$ coincides with the closure of $\Omega$ in $\overline{j_X(C)}$ and hence is compact.
\end{proof}

The following proposition characterizes some separation properties of the Wallman $\kappa$-bounded extension $\Wk X$ of a $T_1$-space.

\begin{proposition}\label{p:WH} For a $T_1$-space $X$ the following statements hold:
 \begin{itemize}
 \item[1)]  $\Wk X$ is Hausdorff iff $X$ is $\overline{\kappa}$-normal;
 \item[2)] $\Wk X$  is Urysohn iff $X$ is strongly $\overline{\kappa}$-normal;
 \item[3)] $\Wk X$ is regular iff $X$ is totally $\overline{\kappa}$-normal.
  \end{itemize}
\end{proposition}

\begin{proof} 1. To prove the ``if'' part of the statement 1), assume that $X$ is $\overline{\kappa}$-normal. Given any distinct closed ultrafilters $u,v\in \Wk X$, use the maximality of $u,v$ and find two disjoint closed sets $F\in u$ and $E\in v$. By definition of $\Wk X$, there exists a subset $C\subset X$ such that $|C|\le\kappa$ and  $u,v\in\overline{j_X(C)}$. By Lemma~\ref{W}, $\overline{C}\in u\cap v$. By the $\overline{\kappa}$-normality of $X$, the disjoint closed sets $F\cap \overline C\in u$ and $E\cap \overline C\in v$ have disjoint open neighborhoods $U$ and $V$ in $X$, respectively. Then $\langle U\rangle$ and $\langle V\rangle$ are disjoints neighborhoods of the ultrafilters $u$ and $v$ in $\Wallman X$, witnessing that the space $\Wk X$ is Hausdorff.

To prove the ``only if'' part, assume that the space $\Wk X$ is Hausdorff. By Proposition~\ref{t:w}, the space $\Wk X$ is $\kappa$-bounded.
To show that the space $X$  is $\overline{\kappa}$-normal, take any subset $C\subset X$ of cardinality $|C|\le\kappa$ and two disjoint closed subsets $F,E$ of $\overline C$. Lemma~\ref{W} implies that $\overline{j_X(F)}\cap \overline{j_X(E)}=\emptyset$. Since $\overline{j_X(F)}\cup \overline{j_X(E)}\subset\overline{j_X(C)}$ and $|C|\leq \kappa$, the sets $\overline{j_X(F)}$ and $\overline{j_X(E)}$ are compact and by the Hausdorffness of $\Wk X$, these compact sets have disjoint open neighborhoods $U$ and $V$ in $\Wk X$. Then $j_X^{-1}(U)$ and $j_X^{-1}(V)$ are disjoint neighborhoods of the sets $F$ and $E$ in $X$, witnessing that the space  $X$ is $\overline{\kappa}$-normal.
\smallskip

2. To prove the ``if'' part of the statement 2), assume that the space $X$ is strongly $\overline{\kappa}$-normal. Given any distinct closed ultrafilters $u,v\in \Wk X$, use the maximality of $u,v$ and find two disjoint closed sets $F\in u$ and $E\in v$. By definition of $\Wk X$, there exists a subset $C\subset X$ such that $|C|\le\kappa$ and  $u,v\in\overline{j_X(C)}$. By Lemma~\ref{W}, $\overline C\in u\cap v$. By the strong $\overline{\kappa}$-normality of $X$, the disjoint closed sets $F\cap \overline{C}\in u$ and $E\cap \overline{C}\in v$ have open neighborhoods $U$ and $V$ in $X$ such that $\overline U\cap \overline V=\emptyset$. Then $\langle U\rangle$ and $\langle V\rangle$ are disjoints open neighborhoods of the ultrafilters $u$ and $v$ in $\Wallman X$. We claim that $\overline{\langle U\rangle}\cap \overline{\langle V\rangle}=\emptyset$. Indeed, given any closed ultrafilter $w\in \Wallman X$, we conclude that either $\overline U\notin w$ or $\overline V\notin w$.
If $\overline U\notin w$, then by the maximality of $w$, the closed set $\overline{U}$ is disjoint with some set in $w$ and then $\langle X\setminus \overline{U}\rangle$ is a neighborhood of $w$, disjoint with $\langle U\rangle$. If $\overline V\notin w$, then $\langle X\setminus\overline V\rangle$ is a neighborhood of $w$ that is disjoint with $\langle V\rangle$. In both cases we obtain that $w\notin \overline{\langle U\rangle}\cap   \overline{\langle V\rangle}$, which implies $ \overline{\langle U\rangle}\cap   \overline{\langle V\rangle}=\emptyset$ and witnesses that the space $\Wk X$ is Urysohn.

To prove the ``only if'' part, assume that the space $\Wk X$ is Urysohn. By Proposition~\ref{t:w}, the space $\Wk X$ is $\kappa$-bounded.
To show that the space $X$ is strongly $\overline{\kappa}$-normal, take any subset $C\subset X$ of cardinality $|C|\le\kappa$ and two disjoint closed subsets $F,E$ of $\overline C\subset X$. Lemma~\ref{W} implies that $\overline{j_X(F)}\cap \overline{j_X(E)}=\emptyset$.
Since $\overline{j_X(F)}\cup\overline{j_X(E)}\subset \overline{j_X(C)}$ and $|C|\leq \kappa$, the sets $\overline{j_X(F)}$ and $\overline{j_X(E)}$ are compact. Since the space $\Wk X$ is Urysohn the compact sets $\overline{j_X(F)}$ and $\overline{j_X(E)}$ have open neighborhoods $U$ and $V$ with disjoint closures in $\Wk X$. Then $j_X^{-1}(U)$ and $j_X^{-1}(V)$ are open neighborhoods with disjoint closures of the sets $F$ and $E$ in $X$, respectively.
\smallskip

3. To prove the ``if'' part of the statement 3), assume that the space $X$ is totally $\overline{\kappa}$-normal. Given any closed ultrafilter $u\in \Wk X$ and a basic open neighborhood $\langle U\rangle$ of $u$ in $\Wk X$, find a closed set $F\in u$ such that $F\subset U$. Since $u\in\Wk X$, there exists a subset $C\subset X$ such that $|C|\le\kappa$ and $u\in\overline{j_X(C)}$. By Lemma~\ref{W}, $\overline{C}\in u$. Replacing the set $F$ by $F\cap\overline C$, we can assume that $F\subset \overline C$. By the total $\overline{\kappa}$-normality of $X$, there exists an open neighborhood $V$ of $F$ in $X$ such that $\overline{V}\subset U$. Using Lemma~\ref{W}, we can show that $u\in\langle V\rangle\subset \overline{\langle V\rangle}\subset \langle U\rangle$,  witnessing the regularity of the space $\Wk X$.

To prove the ``only if'' part, assume that the space $\Wk X$ is regular. By Proposition~\ref{t:w}, the space $\Wk X$ is $\kappa$-bounded.
To show that the space $X$ is totally $\overline{\kappa}$-normal, take any subset $C\subset X$ of cardinality $|C|\le\kappa$ and two disjoint closed subsets $F,E$ of $X$ such that $F\subset \overline{C}$. Lemma~\ref{W} implies that $\overline{j_X(F)}\cap\overline{j_X(E)}=\emptyset$.
Since $\overline{j_X(F)}\subset \overline{j_X(C)}$ and $|C|\leq \kappa$, the set $\overline{j_X(F)}$ is compact. By the regularity of $\Wk X$, the sets $\overline{j_X(F)}$ and $\overline{j_X(E)}$ have disjoint open neighborhoods $U$ and $V$ in $\Wk X$. Then $j_X^{-1}(U)$ and $j_X^{-1}(V)$ are disjoint open neighborhood of the sets $F$ and $E$ in $X$, respectively. Hence $X$ is totally $\overline{\kappa}$-normal.
\end{proof}

The following three theorems give a partial answer to Problem~\ref{prob1} and are the main results of this paper.

\begin{theorem}\label{t:main} For an infinite cardinal $\kappa$ and a $T_1$-space $X$ consider the conditions:
\begin{enumerate}
\item the space $X$ is $\overline\kappa$-normal;
\item the Wallman $\kappa$-bounded extension $W_{\bar\kappa} X$ of $X$ is Hausdorff;
%\item the Wallman $\kappa$-bounded extension $(W_{\overline\kappa} X,j_X)$  is a Hausdorff $\kappa$-bounded reflection of $X$;
\item $X$ is homeomorphic to a subspace of a Hausdorff $\kappa$-bounded space;
\item the space $X$ is  $\overline\kappa$-regular.
\end{enumerate}
Then $(1)\Leftrightarrow(2)\Ra(3)\Ra(4)$. If each closed subspace of  density $\le\kappa$ in $X$ is Lindel\"of, then $(4)\Ra(1)$ and hence the conditions \textup{(1)--(4)} are equivalent.
\end{theorem}

\begin{proof} The equivalence $(1)\Leftrightarrow(2)$ was proved in Proposition~\ref{p:WH}(1) and $(2)\Ra(3)$ follows immediately from Proposition~\ref{t:w} and the fact that the canonical map $j_X:X\to \Wk X$ is a topological embedding. %The equivalence $(2)\Leftrightarrow (3)$ follows from Definition~\ref{d:max} and Theorem~\ref{t:W-max}.
The implication $(3)\Ra(4)$  follows from Proposition~\ref{c1}.
If  each closed subspace of density $\le\kappa$ in $X$ is Lindel\"of, then $(4)\Ra(1)$ by  Proposition~\ref{p:rn}.
 \end{proof}

\begin{theorem}\label{t:main1} For an infinite cardinal $\kappa$ and a $T_1$-space $X$ consider the conditions:
\begin{enumerate}
\item the space $X$ is strongly $\overline\kappa$-normal;
\item the Wallman $\kappa$-bounded extension $W_{\overline\kappa} X$ of $X$ is Urysohn;
%\item the Wallman $\kappa$-bounded extension $(W_{\overline\kappa} X,j_X)$  is a Hausdorff $\kappa$-bounded reflection of $X$;
\item $X$ is homeomorphic to a subspace of a Urysohn $\kappa$-bounded space;
\item $X$ is strongly $\overline\kappa$-regular.
\end{enumerate}
Then $(1)\Leftrightarrow(2)\Ra(3)\Ra(4)$. If each closed subspace of  density $\le\kappa$ in $X$ is countably paracompact in $X$ and Lindel\"of, then $(4)\Ra(1)$ and hence the conditions \textup{(1)--(4)} are equivalent.
\end{theorem}

\begin{proof} The equivalence $(1)\Leftrightarrow(2)$ was proved in Proposition~\ref{p:WH}(2) and $(2)\Ra(3)$ follows immediately from Proposition~\ref{t:w} and the fact that the canonical map $j_X:X\to \Wk X$ is a topological embedding. %The equivalence $(2)\Leftrightarrow (3)$ follows from Definition~\ref{d:max} and Theorem~\ref{t:W-max}.
The implication $(3)\Ra(4)$  follows from Proposition~\ref{c11}.
If each closed subspace of density $\le\kappa$ in $X$ is countably paracompact in $X$ and Lindel\"of, then $(4)\Ra(1)$ by Propositions~\ref{cp} and \ref{p:rn1}.
 \end{proof}

 \begin{theorem}\label{t:main2} For an infinite cardinal $\kappa$ and a $T_1$-space $X$ consider the conditions:
\begin{enumerate}
\item the space $X$ is totally $\overline\kappa$-normal;
\item the Wallman $\kappa$-bounded extension $W_{\overline\kappa} X$ of $X$ is regular;
%\item the Wallman $\kappa$-bounded extension $(W_{\overline\kappa} X,j_X)$  is a Hausdorff $\kappa$-bounded reflection of $X$;
\item $X$ is homeomorphic to a subspace of a regular $\kappa$-bounded space;
\item $X$ is regular.
\end{enumerate}
Then $(1)\Leftrightarrow(2)\Ra(3)\Ra(4)$. If each closed subspace of  density $\le\kappa$ in $X$ is paracompact in $X$, then $(4)\Ra(1)$ and hence the conditions \textup{(1)--(4)} are equivalent.
\end{theorem}

\begin{proof} The equivalence $(1)\Leftrightarrow(2)$ was proved in Proposition~\ref{p:WH}(3), the implication $(2)\Ra(3)$ follows immediately from Proposition~\ref{t:w} and the fact that the canonical map $j_X:X\to \Wk X$ is a topological embedding, and $(3)\Ra(4)$ is trivial. %The equivalence $(2)\Leftrightarrow (3)$ follows from Definition~\ref{d:max} ad Theorem~\ref{t:W-max}.
If each closed subspace of density $\le\kappa$ in $X$ is paracompact in $X$, then $(4)\Ra(1)$ by Propositions~\ref{cp} and \ref{p:rn2}.
 \end{proof}

\begin{problem} Does each $\overline\kappa$-Tychonoff space embed into a Hausdorff $\kappa$-bounded space?
\end{problem}

\section{Some examples}\label{s:ex}
A topological space $X$ is {\em functionally Hausdorff\/} if for any distinct points $x,y\in X$ there exists a continuous function $f:X\to\IR$ such that $f(x)\ne f(y)$. %It is clear that each $\overline\w$-Tychonoff space is functionally Hausdorff.

First, we present an example of a first-countable regular space $M$ which is $\overline{\omega}$-normal but is neither functionally Hausdorff nor strongly $\overline{\omega}$-normal.
The space $M$ is a suitable modification of the famous example of Mysior \cite{Mys}.

Let $\IQ_1=\{y\in\IQ:0<y<1\}$ be the set of rational numbers in the interval $(0,1)$ and $$M=\{-\infty,+\infty\}\cup\IR\cup (\IR\times\IQ_1)$$where $-\infty,+\infty\notin\IR\cup(\IR\times\IQ_1)$ are two distinct points. The topology on the space $M$ is generated by the subbase $$
\big\{\{z\},M\setminus\{z\}:z\in \IR\times\IQ_1\big\}\cup\{V_x:x\in\IR\}\cup\{U_n:n\in\IZ\}\cup\{W_n,n\in\IZ\}$$where $$
\begin{aligned}
V_x&=\{x\}\cup\{(z,y)\in \IR\times\IQ_1:z\in\{x,x+y\}\}\mbox{ for $x\in\IR$},\\
U_n&=\{-\infty\}\cup\{x\in\IR:x<n\}\cup\{(x,y)\in\IR\times\IQ_1:x<n+1\}\mbox{ for $n\in\IZ$},\\
W_n&=\{+\infty\}\cup \{x\in\IR:x>n\}\cup \{(x,y)\in \IR\times\IQ_1:x>n\}\mbox{ for $n\in\IZ$}.
\end{aligned}
$$

\begin{example}\label{exM} The space $M$ has the following properties:
\begin{itemize}
\item[1)] $M$ is regular, first-countable and $\overline{\omega}$-normal;
\item[2)] $M$ is neither functionally Hausdorff nor strongly $\overline{\omega}$-normal.
\end{itemize}
\end{example}

\begin{proof} The definition of the topology of $M$ implies that this space is regular, first-countable and the closure $\overline{C}$ of any countable subset $C\subset M$  is contained in the countable set $$\{-\infty,+\infty\}\cup C\cup\{y,y-z: (y,z)\in C\}.$$
By Proposition~\ref{p:rn} the space $M$ is $\overline{\omega}$-normal.

By analogy with \cite{Mys} (see also \cite{BM} and \cite[1.5.9]{Eng}), it can be shown that $f(-\infty)=f(+\infty)$ for any continuous real-valued function $f$ which means that the space $M$ is not functionally Hausdorff.

Observe that the unit interval $\mathbb{I}=[0,1]$ is a closed discrete subspace of the space $M$.  Besides the discrete topology inherited from $M$, the interval $\II$ carries the standard Euclidean topology, inherited from the real line. The interval $\II$ endowed with the Euclidean topology will be denoted by $\II_E$. To show that the space $M$ is not strongly $\overline{\omega}$-normal we shall need the following fact.

\begin{claim}\label{claim} For any dense subset $A$ in $\II_E$ and any open neighborhood $U$ of $A$ in $M$ the intersection $\overline{U}\cap \II$ is a comeager subset of $\II_E$.
\end{claim}

\begin{proof} To derive a contradiction, assume that the set $B=\mathbb{I}\setminus \overline{U}$
is not meager in $\II_E$ and hence $B$ is of the second Baire category in $\mathbb{I}_E$. 
 Since $B\cap \overline{U}=\emptyset$, for every $b\in B$ there exists a finite subset $F_b$ of $\IQ_1$ and a basic open neighborhood
$$V_{F_b}=\{b\}\cup\{(z,y)\in \IR\times(\IQ_1\setminus F_b):z\in\{b,b+y\}\}$$
of $b$ such that $V_{F_b}\cap \overline U=\emptyset$. For each finite subset $F\subset \IQ_1$ put $B_F=\{b\in B:F_b=F\}$. Since the set of all finite subsets of $\IQ_1$ is countable and $B$ is of the second category, there exists a finite subset $F\subset \IQ_1$ such that the set $B_F$ is not meager in $\mathbb{I}_E$.
Hence there exists an interval $(c,d)\subset \mathbb{I}_E$ such that $B_F$ is dense in $(c,d)$. Recall that $A$ is dense in $\mathbb{I}_E$. At this point it is easy to check that $\emptyset\ne U\cap \bigcup_{b\in B_{F}}V_{F_b}\subset\overline{U}\cap\bigcup_{b\in B_F}V_{F_b}=\emptyset$, which is a desired contradiction.
\end{proof}
Recall that the subspace $\mathbb{I}\subset M$ is discrete. Let $A:=\IQ\cap \mathbb{I}$ and $B:=(\IQ+\sqrt{2})\cap \mathbb{I}$ be two closed countable disjoint subsets of $M$. Assuming that the space $M$ is strongly $\overline\w$-normal, we can find open sets $U_A$ and $U_B$ in $M$ such that $A\subset U_A$, $B\subset U_B$ and $\overline{U_A}\cap\overline{U_B}=\emptyset$. By Claim~\ref{claim}, the sets $\overline{U_A}\cap\II$ and $\overline{U_B}\cap\II$ are comeager in $\mathbb{I}_E$ and hence have nonempty intersection and this is a desired contradiction showing that the space $M$ is not strongly $\overline{\omega}$-normal.
\end{proof}

\begin{remark}
The space $M\setminus\{-\infty,+\infty\}$ is Tychonoff, zero-dimensional, locally compact, locally countable, $\overline{\omega}$-normal but not strongly $\overline{\omega}$-normal.
\end{remark}

Now we present an example of a regular, $\omega$-bounded, totally $\overline{\omega}$-normal space which is not functionally Hausdorff.
Let $[0,\alpha]$ be the ordinal $\alpha+1$ endowed with the order topology. Let $T=[0,\omega_1]{\times}[0,\omega_2]\setminus \{(\omega_1,\omega_2)\}$ be the subspace of the Tychonoff product $[0,\omega_1]{\times}[0,\omega_2]$. Observe that $T$ is $\omega$-bounded. Let $\IZ$ be the discrete space of integers and $-\infty,+\infty$ be distinct points which do not belong to $T\times\IZ$. By $Y$ we denote the set $(T{\times}\IZ)\cup\{-\infty,+\infty\}$ endowed with the topology $\tau$ which satisfies the following conditions:
\begin{itemize}
\item the Tychonoff product $T{\times}\IZ$ is an open subspace in $Y$;
\item if $-\infty\in U\in\tau$, then there exists $n\in\omega$ such that $\{(t,k)\in T\times\IZ:k<-n\}\subset U$;
\item if $+\infty\in U\in\tau$, then there exists $n\in\omega$ such that $\{(t,k)\in T\times\IZ: k>n\}\subset U$.
\end{itemize}
One can check that the space $Y$ is regular and $\omega$-bounded.

On the space $Y$ consider the smallest equivalence relation $\sim$ such that
$(x,\omega_2,2n)\sim (x,\omega_2,2n+1)$ and  $(\omega_1,y,2n)\sim(\omega_1,y,2n-1)$ %, $(\omega_1,y,-2n)\sim (\omega_1,y,-2n-1)$ and $(x,\omega_2,-2n-1)\sim (x,\omega_2,-2n-2)$
for any $n\in \IZ$, $x\in\omega_1$ and $y\in\omega_2$.
Let $X$ be the quotient space $Y/_\sim$ of $Y$ by the equivalence relation $\sim$.

\begin{example}\label{ex} The space $X$ is regular, $\omega$-bounded and totally $\overline\w$-normal, but not functionally Hausdorff and hence is not $\overline{\w}$-Tychonoff.
\end{example}

\begin{proof}
Since the $\omega$-boundedness is preserved by continuous images, the space $X$ is $\omega$-bounded. Using the classical argument due to Tychonoff (see \cite[p.109]{SS}), it can be shown that the space $X$ is regular, but for each real-valued continuous function $f$ on $X$, $f(-\infty)=f(\infty)$. Hence $X$ is not functionally Hausdorff. By Proposition~\ref{p:cn2}, $X$ is totally $\overline{\omega}$-normal.
\end{proof}

\begin{remark}
For each infinite cardinal $\kappa$ the punctured Tychonoff plank $[0,\kappa]{\times}[0,\kappa^+]\setminus \{(\kappa,\kappa^+)\}$ is an example of strongly $\overline{\kappa}$-normal space which is not totally $\overline{\kappa}$-normal.
\end{remark}

A topological space $X$ is called
\begin{itemize}
\item  {\em $H$-compact} if for any open cover $\U$ of $X$ there exists a finite subfamily $\V\subset\U$ such that $X=\bigcup_{V\in\V}\overline V$;
\item {\em $H$-closed} if $X$ is Hausdorff and $H$-compact.
\end{itemize}
It is clear that each compact space is $H$-compact. By \cite[3.12.5]{Eng}, a Hausdorff topological space $X$ is $H$-closed if and only if it is closed in each Hausdorff space containing $X$ as a subspace.
\smallskip

For each infinite cardinal $\kappa$ we shall construct a  $\overline{\kappa}$-normal, $\kappa$-bounded, $H$-compact Hausdorff space which is not Urysohn.
Given an infinite cardinal $\kappa$, denote by $C$ the set of all isolated points of the cardinal $\kappa^+=[0,\kappa^+)$ endowed with the order topology. Write $C$ as the union $C=A\cup B$ of two disjoint unbounded subsets of $\kappa^+$. Choose any points $a,b\notin\kappa^+$ and consider the space $X_{\kappa}=\kappa^+\cup\{a,b\}$ endowed with the topology $\tau$ satisfying the following conditions:
\begin{itemize}
\item $\kappa^+$ with the order topology is an open subspace of $X_{\kappa}$;
\item if $a{\in} U{\in} \tau$, then there exists $\alpha\in\kappa^+$ such that $\{\beta\in A:\beta>\alpha\}\subset U$;
\item if $b{\in} U{\in} \tau$, then there exists $\alpha\in\kappa^+$ such that $\{\beta\in B:\beta>\alpha\}\subset U$.
\end{itemize}

\begin{example}\label{ex1}
For each cardinal $\kappa$ the space $X_{\kappa}$ is $\overline{\kappa}$-normal,  $\kappa$-bounded, $H$-compact and Hausdorff, but not Urysohn.
\end{example}
\begin{proof}
It is straightforward to check that $X_{\kappa}$ is $\overline{\kappa}$-normal, $\kappa$-bounded, and Hausdorff. The $H$-compactness of $X_\kappa$ follows from the observation that for any open neighborhood $U\subset X_\kappa$ of the doubleton $\{a,b\}$ the closure $\overline{U}$ contains the interval $[\alpha,\kappa^+)$ for some ordinal $\alpha\in\kappa^+$.

To see that $X_{\kappa}$ is not Urysohn observe that for any open neighborhoods $U_a$ and $U_b$ of $a$ and $b$, respectively, the sets $\overline{U_a}\cap\kappa^+$ and $\overline{U_b}\cap\kappa^+$ are closed and unbounded in $\kappa^+$. Hence $\overline{U_a}\cap \overline{U_b}\neq\emptyset$.
\end{proof}

Next, we are going to present a (consistent) example of a separable sequentially compact scattered space $X$ which is regular but not $\overline\w$-Tychonoff and hence cannot be embedded into an $\w$-bounded Hausdorff space.

This example is a combination of van Douwen's example \cite[7.1]{vD} of a locally compact sequentially compact space, based on a regular tower, and the famous example of Tychonoff corkscrew due to Tychonoff, see \cite[p.10]{SS}.
First we recall the necessary definitions related to (regular) towers.

By $[\w]^\w$ we denote the family of all infinite subsets of $\w$. For two subsets $A,B\in[\w]^\w$ we write $A\subseteq^* B$ if $A\setminus B$ is finite. Also we write $A\subset^* B$ if $A\subseteq^* B$ but $B\not\subseteq^* A$.
A family $\mathcal T\subset[\w]^\w$ is called a {\em regular tower} if for some regular cardinal $\kappa$ the family $\mathcal T$ can be written as $\mathcal T=\{T_\alpha\}_{\alpha\in\kappa}$ so that
\begin{itemize}
\item[(1)] $T_\beta\subset^* T_\alpha$ for any ordinals $\alpha<\beta$ in $\kappa$, and
\item[(2)] for any $I\in[\w]^\w$ there exists $\alpha\in\kappa$ such that $I\not\subseteq^* T_\alpha$.
\end{itemize}
The first condition implies that the sets $T_\alpha$, $\alpha\in\kappa$, are distinct and hence $\kappa=|\mathcal T|$. Also this condition implies that the relation $\supset^*$ is a well-order on $\mathcal{T}$. %Observe that regular towers are precisely maximal towers of regular length.
% and the enumeration $\Tau=(T_\alpha)_{\alpha\in|\Tau|}$ is unique.

Consider the uncountable cardinals $$
\begin{aligned}
\mathfrak t&=\min\{|\mathcal T|:\mathcal T\subset[\w]^\w\mbox{ \ is a regular tower}\}\\
\hat{\mathfrak t}&=\sup\{|\mathcal T|:\mathcal T\subset[\w]^\w\mbox{ \ is a regular tower}\}
\end{aligned}
$$ and observe that $\mathfrak t\le\hat{\mathfrak t}\le\mathfrak c$. It is well-known that Martin's Axiom implies the equality $\mathfrak t=\hat{\mathfrak t}=\mathfrak c$.

\begin{proposition}\label{p:Zd} The strict inequality $\mathfrak t<\hat{\mathfrak t}$ is consistent. Also $\mathfrak t=\hat{\mathfrak t}=\w_1<\w_2=\mathfrak c$ is consistent.
\end{proposition}

\begin{proof} The consistency of $\mathfrak t=\hat{\mathfrak t}=\w_1<\w_2=\mathfrak c$ was proved in \cite[Theorem 4.1]{BD}.
\smallskip

To prove the consistency of $\mathfrak t<\hat{\mathfrak t}$, assume that {\bf MA}+$\neg${\bf CH} holds in the ground model $V$ and let $V'$
 be the forcing extension of $V$ obtained by adding $\omega_1$ many Cohen reals.
Then $\mathfrak t=\mathfrak b=\omega_1$ in $V'$, which yields a regular tower of
length $\omega_1$ in $V'$. On the other hand, any maximal tower from $V$ of length $(2^\omega)^V>\omega_1$
(which exists, because in $V$, $\mathfrak t=2^\omega>\omega_1$)
remains regular in $V'$ since it is well-known (and easy to check) that Cohen forcing cannot add infinite
pseudointersections to maximal towers. Hence $\mathfrak t<\hat{\mathfrak t}$ in $V'$.
\end{proof}

%Let us recall that a topological space $X$ is {\em scattered} if every non-empty subspace of $X$ has an isolated point.

A topological space $X$ is called {\em $\vec\w$-regular} if for any open set $U\subset X$ and point $x\in U$ there exists a sequence $(U_n)_{n\in\w}$ of open neighborhoods of $x$ such that $\bigcup_{n\in\w}U_n\subset U$ and $\overline{U}_n\subset U_{n+1}$ for all $n\in\w$. It is easy to see that each completely regular space is $\vec\w$-regular.

\begin{example} If $\mathfrak t<\hat{\mathfrak t}$, then there exists a  topological space $X$ such that
\begin{enumerate}
\item $X$ is separable, scattered, and sequentially compact;
\item $X$ is regular but not $\vec\w$-regular and hence not completely regular and not $\overline\w$-Tychonoff;
\item $X$ does not embed into an $\w$-bounded Hausdorff space.
\end{enumerate}
\end{example}

\begin{proof} Since $\mathfrak t<\hat{\mathfrak t}$, there are two regular towers $\mathcal T_1=\{A_{\alpha}\}_{\alpha\in \kappa}$ and $\mathcal T_2=\{B_{\beta}\}_{\beta\in \lambda}$ such that $\kappa<\lambda$. For every $\alpha\in \kappa$ and $\beta\in\lambda$ consider the sets $C_{\alpha}=\omega\setminus A_{\alpha}$ and  $D_{\beta}=\omega\setminus B_{\beta}$. Let $\Tau_1=\{C_{\alpha}\}_{\alpha\in\kappa}$ and $\Tau_2=\{D_{\alpha}\}_{\alpha\in\lambda}$.
Obviously,  $\subset^*$ is a well order on $\Tau_1$ and $\Tau_2$. Also, observe that the families $\Tau_1$ and $\Tau_2$ satisfy the following condition: for any infinite subset $I$ of $\omega$ there exist $C_{\alpha}\in \Tau_1$ and $D_{\beta}\in \Tau_2$ such that the sets $I\cap C_{\alpha}$ and $I\cap D_{\beta}$ are infinite.

For every $i\in\{1,2\}$, consider the space $Y_i=\Tau_i\cup\w$ which is topologized as follows. Points of $\omega$ are isolated and a basic neighborhood of $T\in \Tau_i$ has the form
$$B(S,T,F)=\{P\in\Tau_i\mid S\subset^* P\subseteq^* T\}\cup ((T\setminus S)\setminus F),$$
where $S\in \Tau_i\cup\{\emptyset\}$ satisfies $S\subset^{*}T$ and $F$ is a finite subset of $\omega$.

Repeating arguments of Example 7.1~\cite{vD} one can check that the space $Y_i$ is sequentially compact, separable, scattered and locally compact for every $i\in\{1,2\}$.

For every $i\in\{1,2\}$ choose any point $\infty_i\notin Y_i$ and let $X_i=\{\infty_i\}\cup Y_i$ be the one-point compactification of the locally compact space $Y_i$. It is easy to see that the compact space $X_i$ is scattered, $i\in\{1,2\}$.

Consider the space $\Pi=(X_1\times X_2)\setminus\{(\infty_1,\infty_2)\}$. It is easy to check that the space $\Pi$ is separable, scattered and sequentially compact.

Choose any point $\infty\notin \Pi\times\w$ and consider the space $\Sigma=\{\infty\}\cup(\Pi\times\{\w\})$ endowed with the topology consisting of the sets $U\subset\Sigma$ satisfying two conditions:
\begin{itemize}
\item for any $n\in\w$ the set $\{z\in \Pi:(z,n)\in U\}$ is open in $\Pi$;
\item if $\infty\in U$, then there exists $n\in\w$ such that $\bigcup_{m\ge n}\Pi\times\{m\}\subset U$.
\end{itemize}
Taking into account that the space $\Pi$ is separable, scattered and sequentially compact, we conclude that so is the space $\Sigma$.
On the space $\Sigma$ consider the smallest equivalence relation $\sim$ such that
$(x_1,\infty_2,2n)\sim (x_1,\infty_2,2n+1)$ and $(\infty_1,x_2,2n+1)\sim(\infty_1,x_2,2n+2)$ 
for any $n\in\w$ and $x_i\in X_i\setminus\{\infty_i\}$, $i\in\{1,2\}$.
Let $X$ be the quotient space $\Sigma/_\sim$ of $\Sigma$ by the equivalence relation $\sim$. Observe that the character of the space $X_1$ at $\infty_1$ is equal to the regular cardinal $|\Tau_1|=\kappa$ and is strictly smaller than the pseudocharacter of the space  $X_2$ at $\infty_2$, which is equal to the regular cardinal $|\Tau_2|=\lambda$. Using this observation and repeating the classical argument due to Tychonoff (see \cite[p.109]{SS}), it can be shown that the space $X$ is regular but not $\vec\w$-regular (at the point $\infty$), and hence not Tychonoff and not $\overline\w$-Tychonoff (since for separable $T_1$-spaces the Tychonoff property is equivalent to the $\overline\w$-Tychonoff property). By Proposition~\ref{p:Tych}, the separable space $X$ does not embed into an $\w$-bounded Hausdorff space.
\end{proof}

\begin{question}
Does there exists in ZFC an example of a separable regular sequentially compact space which is not Tychonoff?
\end{question}

\section*{Acknowledgements}
The  authors would like to express their sincere thanks to Lyubomyr Zdomskyy for valuable comments and especially for the idea of the proof of Proposition~\ref{p:Zd}.
%\newpage

\end{document}